\documentclass{amsart}

\usepackage{amsmath, amssymb, amsthm}
\usepackage[latin1]{inputenc}
\usepackage{geometry}
\usepackage[normalem]{ulem}
\usepackage{cancel}
\usepackage[bookmarks=false,colorlinks=true,linkcolor=black,citecolor=black,filecolor=black,urlcolor=black]{hyperref}

\newtheorem{theorem}{Theorem}[section]
\newtheorem{lemma}[theorem]{Lemma}
\newtheorem{proposition}[theorem]{Proposition}

\newtheorem{corollary}[theorem]{Corollary}

\newcommand{\inv}{^{-1}}

\newcommand{\Z}{{\mathbb Z}}
\newcommand{\Q}{{\mathbb Q}}

\date{}
\title{Polynomials defining many units}

\thanks{The first author has been partially supported by CAPES (Proc nº BEX4147/13-8) of Brazil. The second author has been partially supported by Ministerio de Econom\'{\i}a y Competitividad project MTM2012-35240 and Fondos FEDER and Proyecto Hispano-Brasile\~{n}o de Cooperaci\'{o}n Interuniversitaria PHB-2012-0135.}

\author{Osnel Broche}
\address{Osnel Broche, Departamento de Ci\^{e}ncias Exatas, Universidade Federal de Lavras,
Caixa Postal 3037, 37200-000, Lavras, Brazil }
\email{osnel@dex.ufla.br}

\author{\'{A}ngel del R\'{\i}o}
\address{\'{A}ngel del R\'{\i}o, Departamento de Matem\'{a}ticas, Universidad de Murcia,
30100, Murcia, Spain}
\email{adelrio@um.es}

\begin{document}
\begin{abstract}
We classify the polynomials with integral coefficients that, when evaluated on a group element of finite order $n$, define a unit in the integral group ring for infinitely many positive integers $n$. We show that this happens if and only if the polynomial defines generic units in the sense of Marciniak and Sehgal. We also classify the polynomials with integral coefficients which provides units when evaluated on $n$-roots of a fixed integer $a$ for infinitely many positive integers $n$.
\end{abstract}

\maketitle

Let $G$ be a group and let $\Z G$ denote the integral group ring of $G$.
If $x\in G$ and $f$ is a polynomial in one variable with integral coefficients then $f(x)$ denotes the element of $\Z G$ obtained when $f$ is evaluated in $x$.
This paper deals with the problem of when $f(x)$ is a unit of $\Z G$.
It is easy to see that this only depends on the order (possibly infinite) of $x$. 
If the order of $x$ is infinite and $f(x)$ is a unit then necessarily $f=\pm X^m$ for some non-negative integer $m$. Hence we are only interested in the case where $x$ has finite order.
In case $f(x)$ is a unit of $\Z G$ and $n$ is the order of $x$ then we say that $f$ \emph{defines a unit on order} $n$.

Marciniak and Sehgal introduced the following definition \cite{MarciniakSehgal05}.
A polynomial $f$ defines \emph{generic units} if there is a positive integer $D$ such that $f$ defines a unit on every order coprime with $D$.
Marciniak and Sehgal classified the monic polynomials defining generic units.
The first goal of this paper is to complete the classification of polynomials (non-necessarily monic) defining generic units (Theorem~1).
Observe that all the polynomials defining generic units have leading coefficient equal to 1 or -1.

Our second result states that if an integral polynomial $f$ does not define generic units then the set of orders on which $f$ defines generic units is finite (Corollary~\ref{Infinite}) and we give a bound for the cardinality of this set (Corollary~\ref{Bound}). 

%

We could have avoided to give a proof of Theorem~\ref{Generic} because it is a consequence of Corollary~\ref{Infinite}.
However the proof of Theorem~\ref{Generic} is elementary, while the proof of Corollary~\ref{Infinite} uses a deep number theoretical result, namely the $S$-unit Theorem.
The proof of the $S$-unit Theorem ultimately relies on the celebrated Mordell-Weil Theorem.
It would be nice to have an elementary proof of Corollary~\ref{Infinite}.
We would like to thank Hendrik Lenstra for suggesting us the use of the $S$-unit Theorem.
An alternative proof using results by Schinzel on primitive divisors \cite{Schinzel} instead than the S-unit Theorem was suggested to us by B. Mazur.

We will deduce Corollary~\ref{Infinite} and Corollary~\ref{Bound} as consequences of more general results (Theorems~\ref{InfiniteRoots} and \ref{BoundRoots}).
In order to state such results we generalize the definition of polynomials defining units on an order.
Let $f\in \Z[X]$ and let $a$ and $n$ be integers with $n>0$ and $a\ne 0$.
We say that $f$ \emph{defines units on $n$-th roots of} $a$ if $f(x)$ is invertible for every $n$-th root $x$ of $a$ in a ring of characteristic 0 (i.e. $x^n=a$).
If $x$ is an $n$-th root of $a$ in a ring $R$ of characteristic 0 then the map $X\rightarrow x$ induces a ring homomorphism $\Z[X]/(X^n-a)\rightarrow R$.
Thus $f$ defines units on $n$-th roots of $a$ if and only if $f(x)$ is invertible in $\Z[X]/(X^n-a)$, for $x=X+(X^n-a)$.
In particular, $f$ defines units on order $n$ if and only if $f$ defines units on $n$-th roots of $1$.
Theorem~\ref{InfiniteRoots} characterizes the polynomials with integral coefficients that defines units on $n$-th roots of $a$ for infinitely many integers $n$.
For the polynomials $f$ not satisfying this property, Theorem~\ref{BoundRoots} provides a bound for the cardinality of the set of positive integers $n$ for which $f$ defines units on $n$-th roots of $a$.

\medskip

We now establish the basic notation.
If $m$ is a positive integer $m$ then $\zeta_m$ denotes a complex primitive $m$-th root of unity and $\Phi_m$ denotes the $m$-th cyclotomic polynomial, i.e. the minimal polynomial of $\zeta_m$ over the rationals.

We start classifying the polynomials defining generic units. This generalizes \cite[Theorem~3.3]{MarciniakSehgal05}.

\begin{theorem}\label{Generic}
The following conditions are equivalent for a polynomial $f$ in one variable with integral coefficients.
\begin{enumerate}
 \item $f$ defines generic units.
 \item There is an infinite arithmetic sequence $S$ of positive integers such that $f$ defines a unit on every order in $S$.
 \item There is a positive integer $D$ such that for every positive integer $n$, we have that $f$ defines a unit on an order $m$ with $n\le m \le n+D$.
 \item $f=\pm X^m \prod_{i=1}^k \Phi_{m_i}$ with $m$ and $k$ a non-negative integer and $m_1,\dots,m_k$ positive integers which are neither 1 nor prime powers.
\end{enumerate}
\end{theorem}

\begin{proof}
(1) implies (2) and (2) implies (3) are obvious.
(4) implies (1) was proved in \cite[Theorem~3.3]{MarciniakSehgal05}.
To prove (3) implies (4) we follow the lines of the proof of \cite[Theorem~2.1]{MarciniakSehgal05} introducing some modifications.

In the rest of the proof we assume that $f\in \Z[X]$ and $D$ is a positive integer satisfying condition (3) and we have to prove that $f$ satisfies (4).
If $f$ is divisible in $\Z[X]$ by a polynomial $g$ then clearly $g$ also satisfies condition (3) for the same $D$.
Furthermore, condition (4) is closed under taking products.
Therefore we may assume without loss of generality that $f$ is irreducible and $f(X)\neq \pm X$, and we have to prove that $f=\pm \Phi_m$ with $m$ neither 1 nor a prime power.
Observe that the augmentation of $f(x)$ is $f(1)$. Therefore, if $f$ defines generic units then $f(1)=\pm 1$.
However $\Phi_1(1)=0$ and if $p$ is a prime integer then $\Phi_{p^{\alpha}}(1)=p$.
This implies that if $f=\pm \Phi_m$ then $m$ is neither 1 nor a prime power. 
Therefore we only have to show that one root of $f$ is a root of unity and this is equivalent to prove that all the roots of $f$ have modulus 1.

Let $R$ be the set of complex roots of $f(X)$.
Denote
    $$R_0 =\{\alpha\in R: |\alpha|=1\}, \quad R_{-}=\{  \alpha\in R: |\alpha|< 1\}, \quad R_{+}=\{ \alpha\in R: |\alpha|>1\}.$$
Let $\alpha\in R$. We have to prove that $\alpha\in R_0$.
Let $a$ denote the leading coefficient of $f$ and let $d$ be the degree of $f$.

We claim that if $p$ is an integral polynomial of degree $m$ then $a^m p(\alpha)$ is an algebraic integer.
To prove this it is enough to show that $a\alpha$ is an algebraic integer.
Indeed, write $f=aX^d+\sum_{i=0}^{d-1}a_iX^i$ with $a_i\in \Z$ and let $\hat{f}=X^d+\sum_{i=0}^{d-1} a^{d-i-1}a_iX^i.$
Then $\hat{f}$ is an integral monic polynomial and $\hat{f}(a\alpha)=a^{d-1}f(\alpha)=0$.
This proves the claim.

If $x$ is a group element of order $n$ then $\Z C_n \cong \Z[X]/(X^n-1)$.
Therefore for every positive integer $n$ such that $f$ defines a unit on order $n$ there exist $p_n, q_n \in \Z[X]$ such that
$$f\, p_n+(X^n-1)\, q_n=1.$$
We claim that $p_n$ and $q_n$ can be chosen so that the degree of $p_n$ is smaller than $n$ and the degree of $q_h$ is smaller than $d$.
Indeed, as $X^n-1$ is monic, $p_n=r_n+(X^n-1)t_n$ with $r_n$ and $t_n$ integral polynomials such that $r_n$ has degree smaller than $n$.
Let $s_n=q_n+ft_n$. Then $fr_n+(X^n-1)s_n=1$. If $h$ is the degree of $r_n$ and $k$ is the degree of $s_n$ then
$d+h=n+k$. As $h<n$, we have $k<d$. Replacing $p_n$ and $q_n$ by $r_n$ and $s_n$ we obtain the desired conclusion.

Let $h_n$ denote the degree of $q_n$. By the previous paragraph we may assume that $h_n<d$.
Then $a^n(\alpha^n-1)$ and $a^{h_n}q_n(\alpha)$ are algebraic integers.
For every  $\alpha \in R$ we have
$$(\alpha^n-1)q_n(\alpha)=1$$
and therefore
$$\left(\prod_{\alpha\in R} a^{n}(\alpha^n-1)\right) \left(\prod_{\alpha\in R} a^{h_n} q_n(\alpha)\right) =
a^{d(n+h_n)}.$$
Since the two factors of the left side formula are invariant under the action of the Galois group of $\Q(\alpha)$ over $\Q$ of the splitting field of $f$, we deduce that they are rational.
Moreover, they are algebraic integers and hence they are non-zero integers.
We conclude that
	\begin{equation}\label{Cota}
	 \prod_{\alpha\in R} |\alpha^n-1| \le |a|^{dh_n} \le |a|^{d^2}.
	\end{equation}
If $\alpha\in R^-$ then $\lim_{n\rightarrow \infty} \alpha^n=0$ and hence there is a positive real number $\rho$ such that
	\begin{equation}\label{Cota-}
	 \prod_{\alpha\in R^-} |\alpha^n-1|>\rho.
	\end{equation}

Let  $r=|R_0|$ and
let
 $$\delta = \frac{1}{2} \min \{|\alpha^j-1| \; : \; 1 \leqslant j \leqslant
 rD,\;
\alpha \in R_0 \}$$
if $r\neq 0$ and put $\delta =1$ otherwise.
If $\delta=0$ then $R_0$ contains a root of unity, as desired.
So assume $\delta>0$.
Let
	$$T=\{n : f \text{ defines a unit on order } n \text{ and } |\alpha^n-1|\ge \delta \text{ for all } \alpha \in R_0 \}.$$
We claim that $T$ is infinite. Otherwise there is a positive integer $N$ such that $T$ does not contain any number $n$ greater than $N$.
Now we use condition (3) to ensure that there is an increasing sequence of integers $N<n_0<\dots<n_r$ such that $f$ defines a unit on order $n_i$ for each $i=0,1,\dots,r$, and $n_i-n_{i-1}\le D$ for each $i=1,\dots,r$. For every $i=0,1,\dots,r$, we have $n_i\not\in T$ and hence there is $\alpha_i\in R_0$ such that $|\alpha_i^{n_i}-1|< \delta$. Since $|R_0|=r$, necessarily $\alpha_i=\alpha_j$ for some $0\le i < j \le r$.
Therefore there are positive integers $n<m\le n+Dr$ such that $|\alpha^n-1|,|\alpha^m-1|< \delta$.
As $|\alpha|=1$, we have
	$$2\delta\le |\alpha^{m-n}-1|=|\alpha^m-\alpha^n| \le |\alpha^m-1|+|\alpha^n-1|<2\delta,$$
a contradiction.
Thus $T$ is indeed infinite and hence it contains an infinite increasing sequence $(n_k)$.
Thus
\begin{equation}\label{Cota0}
 \prod_{\alpha\in R_0} |\alpha^{n_k}-1| \ge \delta^r
\end{equation}

Applying (\ref{Cota}), (\ref{Cota-}) and (\ref{Cota0}) for each $n_k$ we obtain
	$$\prod_{\alpha\in R^+} |\alpha^{n_k}-1| \le \frac{|a|^{d^2}}{\rho \delta^r}.$$
If $R_+\ne \emptyset$ then the left side of the previous inequality diverges, yielding a contradiction.
Thus $R_+=\emptyset$.

Fix is a group element $x$ of order $n$ with $n>2d$ and such that $f$ defines a unit on order $n$.
It exists by the assumption.
By \cite[Lemma~2.10]{Sehgal}, there exists an integer $j_{n}$ such that  $-\frac{n}{2}< j_n \le \frac{n}{2}$ and
$$f(x\inv)f(x)^{-1}=x^{j_n}.$$
So, if $a_d=a$ then
$$a_dx^{-d}+a_{d-1}x^{-(d-1)}+ \cdots +
a_1x^{-1}+a_0=
a_dx^{d+j_n}+a_{d-1}x^{d-1+j_n}+
\cdots + a_1x^{1+j_n}+a_0x^{j_n}.$$
Then $x^{d+j_n}$ and $x^{j_{n}}$ belong to the support of the element on the right side  and hence they also
belong to the support of the left side.
Consequently, $j_n \le 0$ and $d+j_n=0$. So, $j_n=-d$ and
 $$a_dx^{-d}+a_{d-1}x^{-(d-1)}+ \cdots +
a_1x^{-1}+a_0=
a_d+a_{d-1}x^{-1}+ \cdots +
a_1x^{-d+1} + a_0 x^{-d},$$
or in other words,
$$a_i=a_{d-i},$$ for all $i$. This symmetry
yields that  if $\alpha$ is a complex root of
$f(X)$, then so is $\alpha^{-1}$.
Thus $0=|R_{+}| = |R_{-}|$.
In other words every root of $f$ has modulus $1$.
This implies that all of them are roots of unity.
This finishes the proof.
\end{proof}

Before going on we need a lemma which will be used to characterize when a cyclotomic polynomial defines units on $n$-th roots of $a$.

\begin{lemma}\label{Phipm1}
Let $m$ be a positive integer and $a$ an integer. Then
\begin{enumerate}
    \item $\Phi_m(a)=1$ if and only if one of the following conditions hold:
    \begin{enumerate}
        \item $a=0$ and $m\ne 1$.
        \item $a=1$ and $m$ is neither 1 nor a prime power.
        \item $a=-1$ and $m$ is neither $1$, nor $2$, nor twice a prime power.
        \item $a=2$ and $m=1$.
    \end{enumerate}
    \item $\Phi_m(a)=-1$ if and only if either $a=0$ and $m=1$ or $a=-2$ and $m=2$.
\end{enumerate}
\end{lemma}

\begin{proof}
The result is obvious in case $m=1$ or $2$.
Suppose that $m\ge 3$.
As $\Phi_m=\prod_{\zeta\in R_m} (X-\zeta)$, where $R_m$ is the set of primitive $m$-th roots of unity, if $a\not\in \{-1,0,1\}$ then $|a-\zeta|>1$ and hence $\Phi_m(a)\ne \pm 1$.
Moreover, $\Phi_m(0)=\prod_{\zeta\in R_m} (-\zeta)$ and if $\zeta\in R_m$ then $\zeta\ne \zeta\inv \in R_m$. 
Thus $\Phi_m(0)=1$.
It remains to consider the cases where $a=\pm 1$.

Assume that $a=1$. If $m=p^{\alpha}$ with $p$ prime then $\Phi_m(1)=p$.
Suppose otherwise that $m$ is not a prime power and let $m=p_1^{\alpha_1}\cdots p_k^{\alpha_k}$ be the prime factorization of $m$.
Then $\Phi_m(1)=\Phi_{p_1\cdots p_k}(1)=\Phi_{p_2\dots p_k}(1)\inv \Phi_{p_2\dots p_k}(1)=1$.
Here we have used the following well known facts:
\begin{enumerate}
\item\label{PhiReducion} If $p_1,\dots,p_k$ are distinct primes and $\alpha_i\ge 1$ for each $i=1,\dots,k$ then
    $$
    \Phi_{p_1^{\alpha_1}\dots p_k^{\alpha_k}}(X) = \Phi_{p_1\dots p_k}\left(X^{p_1^{\alpha_1-1}\dots p_k^{\alpha_k-1}}\right).
    $$
\item\label{PhiCoprimes} If $\gcd(h,m)=1$ then
    $$
    \Phi_{hm}(X) = \prod_{d|h} \Phi_m(X^d)^{\mu\left(\frac{h}{d}\right)}.
    $$
where $\mu$ denotes the M\"{o}ebius function.
\end{enumerate}

Suppose that $a=-1$. If $m$ is odd then $\Phi_m(X)=\Phi_{2m}(-X)$ and hence $\Phi_m(-1)=\Phi_{2m}(1)$. 
Thus, in this case $\Phi_m(-1)=1$ unless $m$ is 1. If $m$ is a multiple of $4$ then $\Phi_m(-1)=\Phi_m(1)$, by (\ref{PhiReducion}). Again $\Phi_m(-1)=1$ unless $m$ is a power of $2$. Finally, if $m=2n$ with $n$ odd, then $\Phi_m(-1)=\Phi_n(1)$, which is 1, unless $n$ is either 1 or a prime power.

\end{proof}

The following lemma is obvious.

\begin{lemma}\label{Rewriting}
Let $f$ be an irreducible polynomial in $\Z[X]$ of positive degree, let $\alpha$ be a complex root of $f$ and let $n$ be a positive integer.
Then the following statements are equivalent.
\begin{enumerate}
\item $f$ defines units on $n$-roots of $a$.
\item There are $p,q\in \Z[X]$ such that $pf+q(X^n-a)=1$,
\item $\alpha^n-a$ is a unit in $\Z[\alpha]$.
\end{enumerate}
\end{lemma}

\begin{proposition}\label{CyclotomicDefUnits}
Let $n$ and $m$ be a positive integers and let $a$ be a non-zero integer. Let $d=\frac{m}{\gcd(n,m)}$. Then the following conditions are equivalent:
\begin{enumerate}
\item $\Phi_m$ defines units on $n$-th roots of $a$.
\item $\zeta_d-a$ is a unit in $\Z[\zeta_m]$.
\item $\zeta_d-a$ is a unit in $\Z[\zeta_d]$.
\item $\Phi_d(a)=\pm 1$.
\item One of the following conditions hold:
    \begin{enumerate}
    \item $a=1$ and $d$ is neither 1 nor a prime power.
    \item $a=-1$ and $d$ is neither 1, nor 2, nor twice a prime power.
    \item $a=-2$ and $d=2$.
    \item $a=2$ and $d=1$.
\end{enumerate}
\end{enumerate}
\end{proposition}

\begin{proof}
The equivalence between (1) and (2) is a consequence of Lemma~\ref{Rewriting} (and the fact that $\zeta_d$ and $\zeta_m^n$ are Galois conjugate in $\Q(\zeta_m)$).
(3) obviously implies (2). If (2) holds then $\pm 1=N_{\Q(\zeta_m)/\Q}(\zeta_d-a) = N_{\Q(\zeta_d)/\Q}(\zeta_d-a)^{[\Q(\zeta_m):\Q(\zeta_d)]}$.
Then $N_{\Q(\zeta_d)/\Q}(\zeta_d-a)=\pm 1$ and hence $\zeta_d-a$ is a unit of $\Z[\zeta_d]$. Thus (2) implies (3).

Moreover $\Phi_d=\prod_{r\in U} (X-\zeta_d^r)$ with $U$ the set of positive integers $r< d$ coprime with $d$.
Then  $N_{\Q(\zeta_d)/\Q}(a-\zeta_d)= \prod_{r\in U} (a-\zeta_d^r)=\Phi_d(a)$.
Thus $a-\zeta_d$ is invertible in $\Z[\zeta_d]$ if and only if $\Phi_d(a)=\pm 1$.
This proves that (2) and (4) are equivalent.

The equivalence between (4) and (5) is a consequence of Lemma~\ref{Phipm1}.
\end{proof}

We now classify the polynomials $f$ and the non-zero integers $a$ such that $f$ defines units on $n$-th roots of $a$ for infinitely many positive integers $n$.
We will use multiplicative valuations as defined in \cite{Pierce}.
A place of a field $K$ is an equivalence class of valuations of $K$.
An infinite (respectively, finite) place is an equivalence class of Archimedean (respectively, non-Archimedean) valuations.

\begin{theorem}\label{InfiniteRoots}
Let $f$ be a polynomial in one variable with integral coefficients and let $a$ be a non-zero integer.
Then the following conditions are equivalent:
\begin{enumerate}
\item $f$ defines units on $n$-th roots of $a$ for infinitely many positive integers $n$.
\item $f=\pm X^m\prod_{i=1}^k \Phi_{m_i}$ for $m\ge 0$ and $m_1,\dots,m_k$ positive integers and one of the following conditions hold
\begin{enumerate}
\item $a=1$ and each $m_i$ is neither 1 nor a prime power.
\item $a=-1$ and each $m_i$ is neither 1, nor 2, nor twice a prime power.
\item $a=-2$, $m=0$ and there is $e\ge 1$ such that each $m_i=2^eh_i$ with $h_i$ odd.
\item $a=2$ and and $m=0$.
\end{enumerate}
\end{enumerate}
\end{theorem}

\begin{proof}
(2) implies (1). Assume that $f=\pm X^m\prod_{i=1}^k \Phi_{m_i}$ and $a$ satisfy one of the conditions of (2).
For every $i=1,\dots,k$ let $d_i=\frac{m_i}{\gcd(n,m_i)}$.

If $a=1$ then $f$ defines generic units, by Theorem~\ref{Generic}, and hence it defines units on $n$-th roots of 1 for infinitely many positive integers $n$.

If $a=-1$ then each $m_i$ is neither 1, nor 2, nor twice a prime power. If $n$ is coprime with each $m_i$ then $d_i=m_i$ and hence $\Phi_{m_i}$ defines units on $n$-th roots of $-1$, by Proposition~\ref{CyclotomicDefUnits} for every $i$. As obviously $X$ defines units on $n$-th roots of $-1$ we deduce that $f$ defines units on $n$-th roots of $-1$ for every $n$ coprime with each $m_i$.

Suppose that $a=-2$, $m=0$ and $m_i=2^eh_i$ with $e\ge 1$ and $h_i$ odd. Let $n=2^{e-1}n_1$ with $n_1$ multiple of each $h_i$. Then $d_i=2$ and hence $\Phi_{m_i}$ defines units on $n$-roots of $-2$, by Proposition~\ref{CyclotomicDefUnits}. Hence $f=\pm \prod_{i=1}^k \Phi_{m_i}$ defines units of $n$-th roots of $-2$ for infinitely many positive integers $n$.

Finally assume that $a=2$ and $m=0$. Then for every positive integer $n$ which is multiple of each $m_i$, we have $d_i=1$.
Hence $\Phi_{m_i}$ defines units on $n$-th roots of 2. Thus again $f$ defines units on $n$-th roots of unity for infinitely many positive integers $n$.

(1) implies (2). Assume that $f$ define units on $n$-th roots of $a$ for infinitely many positive integers $n$.
We first prove that every irreducible factor of $f$ is either $X$ or a cyclotomic polynomial.
To prove that we may assume without loss of generality that $f$ is irreducible and different from $X$.
This part of the proof was given to us by Hendrik Lenstra in a private communication.

Let $\alpha$ be a complex root of $f$. We have to prove that $\alpha$ is a root of unity.
Let $M$ be the set of positive integers $n$ such that $f$ defines units on $n$-th roots of $a$.
By Lemma~\ref{Rewriting}, we have $M=\{n\ge 1 : \alpha^n-1 \text{ is a unit in } \Z[\alpha]\}$.
Let $K=\Q(\alpha)$.
As $\alpha$ is algebraic, $K$ is a number field. Let $S$ be the set of places of $K$ formed by all the infinite places and the finite places $v$ for which $v(\alpha)\ne 1$ or $v(a)\ne 1$.
Then $S$ is finite. Let $O_S$ be the subring of $K$ formed by the elements $x$ of $K$ such that $v(x)\le 1$ for every place $v$ which is not in $S$. Then $\alpha$ and $a$ are units of $O_S$ and hence $\Z[\alpha,\alpha\inv]\subseteq O_S$.
If $n\in M$ then $X=\alpha^n/a$ and $Y=1-\alpha^n/a$ is a solution of the equation $X+Y=1$ in the units of $O_S$.
However, the $S$-unit Theorem (see e.g. \cite{EvertseGyoriStewart} or \cite{Lang83}) states that this equation have finitely many solutions in the units of $O_S$.
Thus $\{\alpha^n:n\in M\}$ is finite. As $M$ is infinite, there are $n<m$ in $M$ such that $\alpha^n=\alpha^m$ and hence $\alpha^{m-n}=1$.
This prove that $\alpha$ is a root of unity as desired.

Thus $f=c X^m \prod_{i=1}^{k} \Phi_{m_i}$ for $c\ne 0$, $m\ge 0$, $k\ge 0$ and $m_i$ positive integers.
If $c\ne \pm 1$, then clearly $f$ does not define units on any $n$-th root of an integer. Thus $c=\pm 1$. By Proposition~\ref{CyclotomicDefUnits}, $a$ is $\pm 1$ or $\pm 2$. Moreover, by the same proposition, if $a=1$ then each $m_i$ is neither 1 nor a prime power; if $a=-1$ then each $m_i$ is neither 1, nor 2, nor twice a prime power; and if $a=-2$ then $m_i$ is even. Moreover, if $a=\pm 2$ then $X$ does not define units on $n$-th roots of $a$ for any $n$ and hence in this cases $m=0$.

It remains to prove that if $a=-2$ then all the $m_i$'s has the same valuation at 2. Suppose that $m_1=2^{e_1}h_1$ and $m_2=2^{e_2}h_2$ with $h_1h_2$ odd. If $\Phi_{m_1}\Phi_{m_2}$ defines units on $n$-roots of $-2$ with $n=2^gh$ and $h$ odd then $m_1=2\gcd(n,m_1)$ and $m_2=2\gcd(n,m_2)$. Thus $e_1-\min(e_1,g)=e_2-\min(e_2,g)=1$. Thus $\min(e_1,g)=g=\min(e_2,g)$ and hence $e_1=e_2$.
\end{proof}

The following corollary is a direct consequence of Theorem~\ref{Generic} and the specialization of Theorem~\ref{InfiniteRoots} to $a=1$.

\begin{corollary}\label{Infinite}
The following conditions are equivalent for a polynomial $f$ in one variable with integral coefficients.
\begin{enumerate}
 \item $f$ defines generic units.
 \item $f$ defines units on infinitely many orders.
\end{enumerate}
\end{corollary}


Using a result from \cite{Evertse84} which bounds the number of solutions of an equation of the type $a X+b Y=1$ on $S$-units one can give a bound on the numbers of orders on which $f$ defines a unit.

\begin{theorem}\label{BoundRoots}
Let $f$ be a polynomial in one variable with integral coefficients and let $a$ be non-zero integer.
Write $f=c_dX^d+c_{d+1}X^{d+1} + \dots +c_n X^n$ with $d\le n$ and $c_dc_n\ne 0$.
Let $M$ be the set of positive integers $m$ such that $f$ defines units on $m$-roots of $a$. Suppose that $M$ has finite cardinality.
Let $k$ be the number of primes dividing $ac_dc_n$.
Then
    $$|M| \le 3\cdot 7^{(n-d)(1+2k)}$$
\end{theorem}

\begin{proof}
Suppose that $f_1$ is a divisor of $f$ in $\Z[X]$, with $f(0)\ne 0$.
Let $M_1$ be the set of positive integers $m$ such that $f_1$ defines units on $m$-th roots of $a$ and let $n_1$ and $k_1$ be respectively the degree of $f_1$ and the number of prime divisors of either the leading or independent term of $f_1$.
Then $M\subseteq M_1$, $n_1\le n-d$ and $k_1\le k$. Using this it is clear that it is enough to prove the result under the assumption that $f$ is irreducible and it is neither $X$ nor a cyclotomic polynomial. In particular $d=0$.

Let $\alpha$ be a root of $f$ and let $K=\Q(\alpha)$.
If $v$ is a non-Archimedian valuation of $K$, then $v(c)\le 0$ for every integer $c$. Thus, if $v(\alpha)<1$ then $v(c_0)=v(c_n\alpha^n+\dots +c_1\alpha)\le v(\alpha)<1$ and if $v(\alpha)>1$ then $v(c_n)=v(c_{n-1}\alpha\inv + \dots+ c_1\alpha^{1-n}+c_0\alpha^{-n})\le v(\alpha\inv)<1$. Therefore, every finite places $v$ of $K$ with $v(\alpha)\ne 1$ divides either $c_0$ or $c_n$. Every rational prime is divisible by a maximum of $n$ finite places and therefore if $S$ is the set of places $v$ of $K$ with $v(\alpha)\ne 1$ and $v(a)\ne 1$ then $|S|\le nk$. By the main result of \cite{Evertse84} the number of solutions of $X+Y=1$ on the units of $O_S$ is at most $3\cdot 7^{n+2|S|}\le 3\cdot 7^{n(1+2k)}$.
The proof of Theorem~\ref{InfiniteRoots} shows that then the cardinality of $\{\alpha^n : n\in M\}$ is at most $3\cdot 7^{n+2|S|}\le 3\cdot 7^{n(1+2k)}$.
As $\alpha$ is neither $0$ nor a root of unity, the map $n\rightarrow \alpha^n$ is injective. Thus $|M|\le 3\cdot 7^{n(1+2k)}$.
\end{proof}

\begin{corollary}\label{Bound}
Let $f=c_dX^d+c_{d+1}X^{d+1} + \dots +c_n X^n\in \Z[X]$ with $d\le n$ and $c_dc_n\ne 0$.
Let $k$ be the number of primes dividing $c_dc_n$.
Let $M$ be the set of positive integers $m$ such that $f$ defines a unit on order $m$.
If $f$ does not define generic units then
$$|M| \le 3\cdot 7^{(n-d)(1+2k)}$$
\end{corollary}

\bibliographystyle{amsalpha}
\bibliography{d:/ARTICULO/ReferencesMSC}

\end{document}